\documentclass[11pt]{article}

\usepackage[utf8]{inputenc}
\usepackage[T1]{fontenc}
\usepackage[english]{babel}
\usepackage{amsmath,amssymb, amsthm}
\usepackage{enumitem}
\usepackage{orcidlink}
\usepackage{authblk}

\definecolor{ForestGreen}{rgb}{0.15,0.416,0.18}
\definecolor{EgyptBlue}{rgb}{0.063,0.2,0.65}

\usepackage[noadjust]{cite}
\usepackage{hyperref}

\hypersetup{colorlinks,unicode,
	linkcolor=EgyptBlue,         
	citecolor=EgyptBlue,          
	urlcolor=ForestGreen          
}

\usepackage{geometry}
\usepackage{mathtools}

\newcommand{\doi}[1]{\url{https://doi.org/#1}}

\numberwithin{equation}{section}

\usepackage[capitalize, noabbrev]{cleveref}

\newtheorem{theorem}{Theorem}
\newtheorem{lemma}[theorem]{Lemma}
\newtheorem{proposition}[theorem]{Proposition}
\newtheorem{corollary}[theorem]{Corollary}

\theoremstyle{definition}
\newtheorem{example}[theorem]{Example}

\numberwithin{theorem}{section}
\numberwithin{equation}{section}

\newcommand{\R}{\mathbb{R}}
\newcommand{\N}{\mathbb{N}}

\newcommand{\PP}{\bar{P}^\ast}
\newcommand{\cc}{c^*}
\makeatletter
\newcommand{\subjclass}[2][2020]{%
  \let\@oldtitle\@title%
  \gdef\@title{\@oldtitle\footnotetext{#1 \emph{MSC Classification:} #2}}%
}
\newcommand{\keywords}[1]{%
  \let\@@oldtitle\@title%
  \gdef\@title{\@@oldtitle\footnotetext{\emph{Keywords:} #1}}%
}
\makeatother
\title{Oscillation results for first order neutral delay differential equations with several positive and negative coefficients}

\date{\today}
\author[1,2,3]{\orcidlink{0000-0001-9693-1923}\,Ábel Garab}

\author[1,3]{\orcidlink{0009-0009-3071-020X}\,Gergő Tóth}


\affil[1]{Bolyai Institute, University of Szeged,\par Aradi vértanúk tere 1, Szeged, H–6720, Hungary}

\affil[2]{HUN-REN--SZTE Analysis and Applications Research Group, \par Bolyai Institute,
University of Szeged}

\affil[3]{National Laboratory for Health Security, University of Szeged, Szeged, Hungary}

\begin{document}

\maketitle

\begin{abstract}
    We provide sufficient criteria for the oscillation of all solutions of neutral delay differential equations of the form 
    \[
    \left[x(t) - \sum_{i=1}^{N_r}R_i(t)x(t - r_i(t)) \right]' + \sum_{i=1}^{N_p}P_i(t)x(t - \tau_i(t)) - \sum_{i=1}^{N_q}Q_i(t)x(t - \delta_i(t))=0,
    \] 
    with both positive and negative terms and time-variable delays. 
    Our results improve and generalize several existing criteria available in the literature that address restricted cases, such as constant delays or the absence of negative coefficients. Under additional assumptions on slowly varying parameters, we derive sharper oscillation conditions. We  demonstrate the applicability of our findings through illustrative examples.
\end{abstract}

\section{Introduction}
In this work, we consider linear neutral delay differential equations (NDDEs) with several variable delays of the form
\begin{equation}\label{ndde_var_delays}
    \left[x(t) - \sum_{i=1}^{N_r}R_i(t)x(t - r_i(t)) \right]' + \sum_{i=1}^{N_p}P_i(t)x(t - \tau_i(t)) - \sum_{i=1}^{N_q}Q_i(t)x(t - \delta_i(t))=0, 
\end{equation}
for $t\geq t_0$, where $R_i,\ P_j,\ Q_k,\ r_i,\ \tau_j$, and $\delta_k$ are continuous nonnegative functions, $R_i,\ \tau_j,\ \delta_k\in C^1 $ and $r_i,\ \tau_j,\ \delta_k \leq D\in(0,\infty)$ on the interval $[t_0,\infty)$ for all $1\leq i\leq N_r, \ 1\leq j \leq N_p, \ 1\leq k \leq N_q $, and  $N_q\leq N_p$. Furthermore, we assume that $\delta_i(t)\leq\tau_i(t) $ and   $\delta'_i(t)<1$ for all $1\leq i \leq N_q$ and $t\in(t_0,\infty)$. 

Delay differential equations (DDEs) model systems where the current rate of change depends not only on the present state but also on past states. Among these, NDDEs form a subset in which the delayed term also appears in the derivative.

NDDEs arise in many applications in biology \cite{gopalsamy:he:wen:91,baker:bocharov:paul:rihan:98,huang:ding:wang:25,huang:ding:wang:25-positive,kuang:93}, in physics and engineering \cite{kadar:stepan:23, zhang:stepan:18, rasvan:22, rasvan:73, rasvan:19}, and even in economics \cite{garab:kovacs:krisztin:16,matsumoto:szidarovszky:10,balazs:krisztin:19}, to mention only a few examples.

For a thorough introduction to the theory of DDEs, we refer to the books \cite{hale:verduyn-lunel:93} and \cite{diekmann:van-gils:verduyn-lunel:walther:95}. For a more applied approach, we recommend that the reader study \cite{Erneux:09} and \cite{Smith:11}.

The oscillation theory of DDEs has received considerable attention in the last sixty to seventy years. The monographs \cite{gyori1991, Erbe1995, agarwal2004} yield a nice overview of the main contributions to the field from the last century. Each of these books devotes a complete chapter to oscillation  and nonoscillation criteria for NDDEs.

We say that a function is a solution of \eqref{ndde_var_delays} if it is continuous on $[t_0-D,\infty)$, continuously differentiable on the interval $(t_0,\infty)$, and satisfies \eqref{ndde_var_delays} there. A solution of a scalar differential equation oscillates if it has arbitrarily large zeros. A differential equation oscillates if every solution of the equation oscillates. A solution $x$ of \eqref{ndde_var_delays} is eventually positive if there exists a $T>t_0$ such that $x(t)>0$ for all $t>T$. A  linear equation is oscillatory if and only if it has no eventually positive solution.

Several studies have addressed oscillation criteria for special cases of equation \eqref{ndde_var_delays}. For instance, \cite[Section 6.9]{gyori1991} considers time-variable delays and coefficient functions, but no negative terms are allowed (i.e. $Q_k\equiv 0$ for all $k$), while \cite[Sections 3.3--3.5]{agarwal2004} involves both a single positive and a negative coefficient function (more precisely, $N_r=N_p=N_q=1$), and constant delays. 

There is still considerable progress related to the oscillation properties of equation \eqref{ndde_var_delays}. For instance, \cite{attia:24} studies the distance between consecutive zeros of solutions of similar equations; there are recent papers on the oscillation of dynamic equations \cite{grace:negi:abbas:22} and nonlinear equations of a similar form \cite{pinelas:santra:18}. However, very little is known about the oscillation of \eqref{ndde_var_delays} in its full generality. The aim of this work is to fill this gap.
We mention here the results of \cite{parhi:chand:00,luo:shen:04,elabbasy:hassan:saker:07} that allow both positive and negative coefficient functions, but only constant delays. Their approach and the obtained oscillation criteria are not comparable to ours. 

Our first main result (\Cref{c_thm:ndde-rdde}) generalizes the results of Győri and Ladas \cite[Section~6.9]{gyori1991}. It establishes that the NDDE \eqref{ndde_var_delays} is oscillatory, provided a corresponding first-order, non-neutral DDE is oscillatory. We combine this with classical oscillation criteria for the latter equation to obtain explicit criteria for the oscillation of \eqref{ndde_var_delays}. Our second main result (\Cref{thm:agar}) is based on a comparison-type result (\Cref{thm:eq-ineq}) and an iterative technique. It yields a generalization of a result by Agarwal, Bohner, and Li \cite[Theorem 3.4.9] {agarwal2004} for multiple constant delays. We demonstrate that both main theorems can be sharpened when a property related to slow variation is fulfilled (see \Cref{cor:slow,cor:slow-osc}). 

The rest of the paper is structured as follows. In the next section, we clarify some further notations and prove two auxiliary results that are used throughout the paper. \Cref{sec:general-results} contains our main results related to the general case, when the delays may be time-variable. \Cref{sec:constant-results} is devoted to our results related to the case of constant delays. We conclude the paper with some illustrative examples in \Cref{sec:examples}, where we also discuss the relation of the obtained oscillation criteria.

\section{Preliminaries} 
In this section, we set some notations and prove two auxiliary results. As customary, let $\N$ and $\R$ denote the set of nonnegative integers and reals, respectively. 

The next simple proposition is needed for \Cref{c_lem:y_pos}, which plays a key role in the proof of all our results.

\begin{proposition}\label{prop:c_def}
    Assume that $\tau$ and $\delta$ are both continuously differentiable functions on $[t_0,\infty)$ with $0 \leq \delta(t)\leq \tau(t)\leq D$ for some positive constant $D$, such that $\delta'(t)<1$ for all $t\geq t_0$. Then for any $t\geq t_0$, there exists a unique $c(t)\in [0,D]$ such that
    \begin{equation}\label{c_def}
        c(t)=\tau(t)-\delta(t-c(t)).    
    \end{equation}
    Moreover, the function $t\mapsto c(t)$ is differentiable and fulfills
    \begin{equation}      \label{c'}
        1-c'(t) = \frac{1-\tau'(t)}{1-\delta'(t-c(t))}.
    \end{equation}
\end{proposition}
\begin{proof}
    Define the following continuous function
    \begin{equation}\label{F_def}
        F\colon \R\times [t_0,\infty) \to \R, \qquad          F(c,t)\coloneqq c-\tau(t)+\delta(t-c).
    \end{equation}
    Since $F(0,t)\leq0$, $F(D,t)\geq0$, and
    \[
        \frac{\partial F}{\partial c}(c,t)=1-\delta'(t) > 0,
    \]
    hence for any $t\geq t_0$, there exists a unique $c\in [0,D]$ satisfying \eqref{c_def}. Applying the implicit function theorem, we obtain that $c$ is differentiable and $F(c(t),t)=0$. Differentiating this with respect to $t$ yields \eqref{c'}.
\end{proof}

\subsubsection*{Notation}
Recall that $N_p\geq N_q$. For the convenience of notation, we define $Q_i\equiv0,\ \delta_i\equiv0 $, for all $N_q < i \leq N_p$. 

In view of \cref{prop:c_def}, for every pair $(\delta_i,\tau_i)$ ($i\in \{1,2,\dots,N_p\}$), let $c_i$ denote the unique differentiable function that satisfies 
\begin{equation}\label{c-i-def}
c_i(t)=\tau_i(t)-\delta_i(t-c_i(t)), \quad \text{and}\quad  c_i(t)\in[0,D]\quad \text{for all } t\geq t_0+D.
\end{equation}
Note that $c_i=\tau_i$ for $N_q<i\leq N_p$. Furthermore,  for each $1\leq i\leq N_p$, introduce the function 
\begin{equation}\label{c_P_bar}
    \bar{P}_i(t)\coloneqq  P_i(t) - Q_i(t - c_i(t))(1 - c_i'(t)), \qquad  t\geq t_0+D.
\end{equation}

The following lemma provides a convenient substitution that will be used several times in the sequel. It generalizes both \cite[Lemma 6.9.1.]{gyori1991} and \cite[Lemma 3.3.1]{agarwal2004} by allowing both multiple  negative coefficient functions and several time-variable delays. 
\begin{lemma}\label{c_lem:y_pos}
    Let $c_i$ and $\bar{P}_i$ be defined by \eqref{c-i-def}--\eqref{c_P_bar}, and $\bar{P}_i(t) \geq0$ for all $t\geq t_0+D$. Suppose that $x$ is an eventually positive solution of \eqref{ndde_var_delays}, and that the following properties hold:
    \begin{equation}\label{c_eq:P_bar_weakpos}
    \forall T_0\in{(t_0,\infty)}\; \exists \xi>T_0 :\ \sum\limits_{i=1}^{N_p}\bar{P}_i(\xi)>0,
    \end{equation}
    \begin{equation}\label{c_eq:key_ineq}
    \sum_{i=1}^{N_r}R_i(t)+\sum_{i=1}^{N_q}\int_{t-c_i(t)}^{t} Q_i(s)\,ds \leq 1.
    \end{equation}
    Define
    \begin{equation}\label{c_eq:y_def}
        y(t)\coloneqq x(t)-\sum_{i=1}^{N_r}R_i(t)x(t-r_i(t))-\sum_{i=1}^{N_q}\int_{t-c_i(t)}^tQ_i(s)x(s-\delta_i(s))\,ds.
    \end{equation}
    Then $y'(t)\leq0$ and $x(t)\geq y(t) > 0 $.
\end{lemma}
The proof is similar to that of \cite[Lemma 3.3.1]{agarwal2004}.
\begin{proof}
    Assume that $x(t)$ is a solution of \eqref{ndde_var_delays} such that $x(t)> 0$ for all $t\geq T\geq t_0+D$. Differentiating $y(t)$ yields
    \begin{align*}
        y'(t) ={}& \biggl[x(t) - \sum_{i=1}^{N_r} R_i(t) x(t - r_i(t))\biggr]' - \sum_{i=1}^{N_q} Q_i(t) x(t - \delta_i(t)) \\ 
        &{}+\sum_{i=1}^{N_q} Q_i(t - c_i(t)) x\Big(t - c_i(t)-\delta_i(t-c_i(t))\Big) (1 - c_i'(t)).
    \end{align*}
    In view of equation \eqref{ndde_var_delays} and the definition of $\bar{P}_i$ and $c_i$ we obtain
    \[
        y'(t) = - \sum_{i=1}^{N_p} \Big[ P_i(t) - Q_i(t - c_i(t))(1 - c_i'(t)) \Big] x(t - \tau_i(t)),
    \]
    or equivalently
    \begin{equation}\label{c_eq:y'=}
        y'(t)=-\sum_{i=1}^{N_p}\bar{P}_i(t)x(t-\tau_i(t)).
    \end{equation}
    Therefore $y'(t)\leq0$ for $t\geq T$.
    It remains to prove the positivity of $y(t)$. In view of \eqref{c_eq:P_bar_weakpos}, \eqref{c_eq:y'=} and the positivity of $x(t)$, it suffices to prove that $y(t)$ is nonnegative.
    
    We distinguish two cases. 

    \medskip \noindent
    \textbf{Case 1.} Suppose that $x(t)$ is unbounded. Then there exists a sequence $(s_n)_{n \in \N}$ with $s_n \to \infty$, $x(s_n)\to\infty$ such that
    \[
        x(t) \leq x(s_n) \quad \text{for all } t \leq s_n.
    \]
    Hence,
    \begin{align*}
        y(s_n) &= x(s_n) - \sum_{i=1}^{N_r} R_i(s_n) x(s_n - r_i(s_n)) - \sum_{i=1}^{N_q} \int_{s_n - c_i(s_n)}^{s_n} Q_i(s) x(s - \delta_i(s)) \, ds\\
       &\geq x(s_n) \left(1 - \sum_{i=1}^{N_r}R_i(s_n) - \sum_{i=1}^{N_q}\int_{s_n -c_i(s_n)}^{s_n} Q_i(s) \, ds \right) \geq 0
    \end{align*}
    for all $n$ by property \eqref{c_eq:key_ineq}. By the monotonicity of $y$, we conclude that $y(t) \geq 0$. 

    \medskip \noindent
    \textbf{Case 2.} Suppose that $x(t)$ is bounded. Let
    \[
        M_x \coloneqq \limsup_{t \to \infty} x(t) \quad \text{and} \quad M_y \coloneqq \lim_{t \to \infty} y(t) \in \R \cup \{-\infty\}.
    \]
    Then there exists a sequence $(s_n)_{n \in \N}$ with $s_n \to \infty$ and $x(s_n) \to M_x$.
    Moreover, let $(\xi_n)_{n \in \N}$ be such that
    \[
        x(\xi_n) \coloneqq  \max \{ x(s) : s_n - 2D \leq s \leq s_n \}.
    \]
    Clearly, $\xi_n \to \infty$, and, by passing to a convergent subsequence if necessary (without changing notation), we may assume
    \[
        \lim_{n \to \infty} x(\xi_n) \leq M_x.
    \]
    
    Now, consider the difference
    \[
        x(s_n) - y(s_n) = \sum_{i=1}^{N_r} R_i(s_n) x(s_n - r_i(s_n)) + \sum_{i=1}^{N_q} \int_{s_n - c_i(s_n)}^{s_n} Q_i(s) x(s - \delta_i(s)) \, ds.
    \]
    By the definition of $\xi_n$, we have
    \[
        x(s_n) - y(s_n) \leq x(\xi_n)\left(\sum_{i=1}^{N_r} R_i(s_n) + \sum_{i=1}^{N_q}\int_{s_n - c_i(s_n)}^{s_n} Q_i(s) \, ds \right).
    \]
    Using inequality \eqref{c_eq:key_ineq} we obtain
    \[
        x(s_n) - y(s_n) \leq x(\xi_n),\quad \ \text{for all } n \in \N.
    \]
    Letting $n\to \infty$ yields $M_x - M_y \leq M_x$, and thus $M_y \geq 0.$
    Since $y'(t) \leq 0$, it follows that $y(t) \geq M_y \geq 0, \text{ for } t\ge T $.
\end{proof}

\section{Results -- variable delays} \label{sec:general-results}
First, we consider the most general case, where all the delay functions in \eqref{ndde_var_delays} are time-dependent and satisfy the properties listed in the introduction. Now we will connect our equation to a simple retarded delayed differential equation. This provides a generalization of \cite[Theorem 6.9.1]{gyori1991}.
\begin{theorem}\label{c_thm:ndde-rdde}
    Let $\PP_i(t)$ and $\cc_i(t)$ be arbitrary such that $0\leq\PP_i(t)\leq \bar{P}_i(t)$ and $\cc_i(t)\geq c_i(t)$  for all $t\in[t_0+D,\infty)$ and $1\leq i \leq N_p$.
    Furthermore, assume that they satisfy
    \begin{align}
            \forall T_0\in{(t_0,\infty)}\; \exists \xi>T_0 :\ \sum\limits_{i=1}^{N_p}\PP_i(\xi)&>0,  \label{H1}\\
        \shortintertext{and}\sum_{i=1}^{N_r}R_i(t)+\sum_{i=1}^{N_q}\int_{t-\cc_i(t)}^{t} Q_i(s)\,ds &\leq 1.  \label{H2}
    \end{align}
    If every solution of     \begin{equation}\label{c_eq:retarded}
        x'(t) + \sum_{i=1}^{N_p} \PP_i(t) x(t - \tau_i(t)) = 0
    \end{equation}
    oscillates, then every solution of equation \eqref{ndde_var_delays} also oscillates.
\end{theorem}
\begin{proof}
    Assume that every solution of \eqref{c_eq:retarded} oscillates. For the sake of contradiction, suppose that \eqref{ndde_var_delays} has an eventually positive solution $x$. Notice that conditions \eqref{H1}--\eqref{H2} imply that inequalities \eqref{c_eq:P_bar_weakpos} and \eqref{c_eq:key_ineq} hold. Define $y$ by \eqref{c_eq:y_def}. As seen in the proof of \Cref{c_lem:y_pos}, 
    \[
        y'(t)+\sum_{i=1}^{N_p}\bar{P}_i(t)x(t-\tau_i(t))=0.
    \]
    Using the fact that $y(t)\leq x(t)$, and $\PP_i(t)\leq \bar{P}_i(t)$ we obtain
    \[
        y'(t)+\sum_{i=1}^{N_p}\PP_i(t)y(t-\tau_i(t))\leq0,
    \]
    where $y$ is eventually positive by \Cref{c_lem:y_pos}. According to \cite[Corollary 3.2.2]{gyori1991}, this inequality implies that there exists an eventually positive solution $z$ of
    \[
        z'(t)+\sum_{i=1}^{N_p}\PP_i(t)z(t-\tau_i(t))=0.
    \]
    This is a contradiction that completes our proof.
\end{proof}
Since, in general, we do not have explicit formulas for the functions $c_i$, a direct application of the previous theorem and corollaries might be difficult. Therefore, we provide further sufficient conditions. Let us recall that $D$ is an upper bound for all delays.
\begin{corollary}\label{cor:ndde-rdde}
    Assume that there exists $\Delta_j$ such as $\delta'_j(t)\leq \Delta_j<1$ and $\tau'_i(t)<1$ for all $t\geq t_0,\ j\in\{1,\dots N_q\},\ i\in\{1,\dots,N_p \} $. Let
    \[
        \PP_i(t)\coloneqq P_i(t)-\frac{1-\tau'_i(t)}{1-\Delta_i}\sup_{s\in[t-D,t]}Q_i(s)>0, \qquad (1\leq i\leq N_p) 
    \]
    and
    \begin{equation}\label{sufficient-for-H2}
        \sum_{i=1}^{N_r}R_i(t)+\sum_{i=1}^{N_q}\int_{t-D}^{t} Q_i(s)\,ds \leq 1 ,
    \end{equation}
    moreover, assume that every solution of
    \[
        x'(t)+\sum_{i=1}^{N_p}\PP_i(t)x(t-\tau_i(t))=0
    \]
    oscillates. Then every solution of \eqref{ndde_var_delays} oscillates.
\end{corollary}
\begin{proof}
    Inequality \eqref{sufficient-for-H2} implies  \eqref{H2}, while condition \eqref{H1} comes from the inequalities
    \begin{align*}
        0< P_i(t)-\frac{1-\tau'_i(t)}{1-\Delta_i}\sup_{s\in[t-D,t]}Q_i(s)&\leq P_i(t)-\frac{1-\tau_i'(t)}{1-\delta_i'(t-c_i(t))}Q_i(t-c_i(t)) \\
        &= P_i(t)-(1-c_i'(t))Q_i(t-c_i(t)) =\bar{P}_i(t).
    \end{align*}
    The claim follows immediately by \Cref{c_thm:ndde-rdde}.
\end{proof}
\Cref{c_thm:ndde-rdde} or \Cref{cor:ndde-rdde} can be combined with any oscillation criterion for first order linear DDEs of the form $x'(t)+\sum_{i=1}^n p_i(t)x(t-\tau_i(t))=0$ to obtain concrete sufficient criteria for the oscillation of all solutions of equation \eqref{ndde_var_delays}. The interested reader may find numerous such criteria in the survey papers \cite{moredi:stavroulakis:18-survey, stavroulakis:14-survey} or some more recent ones e.g.\ in \cite{abel2020,koplatadze:24,attia:el-matary:23,dix:21,chatzarakis:jadlovska:18, pituk:stavroulakis-john:22, pituk:stavroulakis:stavroulakis:23, zhuang:wang:wu:21,zhuang:wang:wu:21-iteration} and in the references therein. In the following corollary we combine \Cref{c_thm:ndde-rdde} with some classical results. 
\begin{corollary}\label{cor:I}
    Assume that there exist functions $0\leq\PP_i\leq \bar{P}_i$ and $\cc_i\geq c_i$ such that conditions \eqref{H1}--\eqref{H2} hold. Let $\tau_{\min}(t)\coloneqq \min_{1\leq i\leq N_p}\tau_i(t)$. If any of the following conditions is fulfilled, then every solution of equation \eqref{ndde_var_delays} oscillates.
    \begin{equation*}\tag{$A_1$}\label{A_1}
        \tau_i(t)\equiv\tau_i\quad \text{are positive constants and}\quad \liminf_{t\to\infty}\sum_{i=1}^{N_p}\int_{t-\tau_i}^t\PP_i(s)\,ds > \frac{1}{e},
    \end{equation*}
    \begin{equation*}\tag{$B_1$}\label{B_1}
    \liminf_{t\to\infty}\sum_{i=1}^{N_p}\PP_i(t)\tau_i(t)>\frac{1}{e},
    \end{equation*}
    \begin{equation*}\tag{$C_1$}\label{C_1}
        \liminf_{t\to\infty}\int_{t-\tau_{\min}(t)}^t\sum_{i=1}^{N_p}\PP_i(s)\,ds > \frac{1}{e},
    \end{equation*} 
    \begin{equation*}\tag{$D_1$}\label{D_1}
        \limsup_{t\to\infty}\int_{t-\tau_{\min}(t)}^t\sum_{i=1}^{N_p}\PP_i(s)\,ds > 1.
    \end{equation*}
\end{corollary}
\begin{proof}
    According to \cite[Corollary~1]{bingtuan1996}, \cite[Corollary~2.6.2]{Erbe1995}, \cite[Theorem~2.7.1]{ladde_lak1987}, and \cite[Remark~2.7.3]{ladde_lak1987}, respectively, each of conditions \eqref{A_1}, \eqref{B_1}, \eqref{C_1}, and \eqref{D_1} ensures that every solution of \eqref{c_eq:retarded} oscillates. 
    Hence, by \Cref{c_thm:ndde-rdde}, it follows that every solution of \eqref{ndde_var_delays} oscillates.
\end{proof}

The next corollary uses the so-called slowly varying property. A function $f$ is slowly varying at infinity (additively) if, for all $h\in\R$, the expression $|f(t+h)-f(t)|$ tends to zero as $t$ goes to infinity. We recall a characterization of continuous slowly varying functions given by \cite[p.\ 30]{Pituk2017}: a continuous function $f$ is slowly varying at infinity if and only if it can be decomposed into a sum of two functions, where one of them is continuous and converges to a finite value, and the other one is differentiable with its derivative tending   to zero as $t\to\infty$.
\begin{corollary}\label{cor:slow}
    Assume that there exist functions $0\leq\PP_i\leq \bar{P}_i$ and $\cc_i\geq c_i$ that conditions \eqref{H1}--\eqref{H2} hold. Moreover, suppose that $\liminf\limits_{t\to\infty}\int_{t-\tau_i(t)}^t\PP_i(s)\,ds>0$, $1\leq i \leq N_p $, 
    $\PP_i(t)$, and $\tau_i(t)$ are all uniformly continuous and bounded for all $i\in\{1,2,\dots,N_p\}$. If any of the following conditions is fulfilled, then every solution of equation \eqref{ndde_var_delays} oscillates.
   \begin{enumerate}[label=\textnormal{(\alph*)}]
        \item
        $\tau_i(t)\equiv \tau_i$ are positive constants, $\sum_{i=1}^{N_p}\int_{t-\tau_i}^{t}\PP_i(s)\,ds$ is slowly varying at infinity and
        \begin{equation*}\label{A_2}\tag{$A_2$}       
        \limsup_{t\to\infty}\sum_{i=1}^{N_p}\int_{t-\tau_i}^{t}\PP_i(s)\,ds > \frac{1}{e}.
        \end{equation*}
        
        \item
        $\sum_{i=1}^{N_p}\PP_i(t)\tau_i(t)$ is slowly varying at infinity and
        \begin{equation*}\label{B_2}\tag{$B_2$}
        \limsup_{t\to\infty}\sum_{i=1}^{N_p}\PP_i(t)\tau_i(t) > \frac{1}{e}.
        \end{equation*}
        
        \item
        There exists a uniformly continuous function $\hat \tau\colon [t_0,\infty)\to [0,\infty)$ such that  $0\leq \hat\tau(t)\leq \tau_i(t)$ for all $t\geq t_0$ and $1\leq i\leq N_p$, and  $\int_{t-\hat\tau(t)}^{t}\sum_{i=1}^{N_p}\PP_i(s)\,ds $ is slowly varying at infinity, and
        \begin{equation*}\label{C_2}\tag{$C_2$}
        \limsup_{t\to\infty}
        \int_{t-\hat{\tau}(t)}^{t}\sum_{i=1}^{N_p}\PP_i(s)\,ds > \frac{1}{e}.
        \end{equation*}
    \end{enumerate}
\end{corollary}

\begin{proof}
    According to \cite[Theorem 2.2. (a)/(b)/(c)]{abel2020}, either \eqref{A_2}, \eqref{B_2}, or \eqref{C_2} implies that every solution of \eqref{c_eq:retarded} oscillates, thus applying \Cref{c_thm:ndde-rdde} we obtain that every solution of \eqref{ndde_var_delays} oscillates.
\end{proof}
There are special cases in which the functions $c_i(t)$ admit explicit closed-form expressions -- see \Cref{ex-variable-delays}. In particular, this is the case, if the functions $\delta_i$ are all constants.

Indeed, if the delay functions $\delta_i$ are all constant, then $c_i(t)$ can be expressed explicitly as $c_i(t)=\tau_i(t)-\delta_i$ and \eqref{c_P_bar} simplifies to
\begin{equation}\label{eq:P_bar}
    \bar{P}_i(t)\coloneqq  P_i(t) - Q_i(t - \tau_i(t)+ \delta_i)(1 - \tau_i'(t)), \qquad  t\in[t_0+D,\infty).
\end{equation}
Our results so far can be sharpened as summed up in the following theorem.
\begin{theorem}\label{thm:constant-delta}
    If the delays $\delta_k$ are constant for $1\leq k\leq N_q$, then \Cref{c_thm:ndde-rdde,cor:I,cor:slow} hold true with the choice 
    \[\bar{P}^*_i(t)\coloneqq P_i(t) - Q_i(t - \tau_i(t)+ \delta_i)(1 - \tau_i'(t)) \quad \text{and}\quad c_i^*(t)\coloneqq \tau_i(t)-\delta_i.\]
\end{theorem}

\section{Results -- constant delays}\label{sec:constant-results}
This section is devoted to the special case when all the delays in \eqref{ndde_var_delays} are constant, that is, we deal with the NDDE
\begin{equation}\label{eq:constant_delays}
    \bigg[x(t)-\sum_{i=1}^{N_r}R_i(t)x(t-r_i)\bigg]'+\sum_{i=1}^{N_p}P_i(t)x(t-\tau_i)-\sum_{i=1}^{N_q}Q_i(t)x(t-\delta_i)=0, 
\end{equation}
for $t\geq t_0$, where $N_q\leq N_p$, the coefficient functions $R_i,\ P_j,$ and $ Q_k$ are  nonnegative and continuous for all $1\leq i\leq N_r$, $1\leq j\leq N_p$ and $1\leq k\leq N_q$, while $R_i\in C^1$, and the constant delays $r_i,\ \tau_j,\ \delta_k $ are nonnegative, bounded above by $D$ and fulfill $\delta_k\leq \tau_k$ for all $1\leq k\leq N_q$. 

In this case \eqref{eq:P_bar} simplifies to $\bar{P}_i(t)=P_i(t)-Q_i(t-\tau_i+\delta_i)$. The following comparison theorem generalizes \cite[Theorem 3.4.1]{agarwal2004} by allowing multiple delays. It also serves as a basis to our second main result, \Cref{thm:agar}.

\begin{theorem}\label{thm:eq-ineq}
    Assume that $\tau_i>0$ for all $1\leq i \leq N_p$, moreover,
    \begin{align*}
        \bar{P}_i(t)=P_i(t)-Q_i(t-\tau_i+\delta_i)&\geq 0,\\
        \sum_{i=1}^{N_r}R_i(t)+\sum_{i=1}^{N_q}\int_{t-\tau_i+\delta_i}^{t} Q_i(s)\,ds &\leq 1,
    \end{align*}
    and either
    \begin{align}\label{eq:R+P}
        \sum_{i=1}^{N_r}R_i(t)+\sum_{i=1}^{N_p}\bar{P}_i(t)&>0 \\
    \shortintertext{or}    
   \label{eq:P>0fortau}
        \sum_{i=1}^{N_p}\int_{t-\tau_i}^t\bar{P}_i(s)\,ds&>0
    \end{align}
    for all $t\geq t_0+D$. Then every solution of equation \eqref{eq:constant_delays} oscillates if and only if every solution of the corresponding differential inequality 
    \begin{equation}\label{eq:ineq}
    \bigg[x(t)-\sum_{i=1}^{N_r}R_i(t)x(t-r_i)\bigg]'+\sum_{i=1}^{N_p}P_i(t)x(t-\tau_i)-\sum_{i=1}^{N_q}Q_i(t)x(t-\delta_i)\leq 0
    \end{equation}
    oscillates.
\end{theorem}

\begin{proof}
    The proof follows the main ideas of \cite[Theorem 3.4.1]{agarwal2004}. For the sake for completeness, we provide the details here.
    
    The sufficiency part is trivial. We will prove the necessity. Suppose that $x$ is an eventually positive solution of \eqref{eq:ineq}. Define $y$ as in \Cref{c_lem:y_pos} using constant delays. A similar argument to the one presented there shows
    \[
    y'(t)\leq-\sum_{i=1}^{N_p}\bar{P}_i(t)x(t-\tau_i)
    \]
    for sufficiently large $t$. Integrating on $ [t,\infty)$ yields
     \[
     y(t)\geq \int_t^{\infty}\sum_{i=1}^{N_p} \bar{P}_i(s)x(s-\tau_i)\, ds.
     \]
    From the definition of $y$ we obtain that
    \begin{equation}\label{eq:x_large}
        x(t)\geq\sum_{i=1}^{N_r}R_i(t)x(t-r_i)+\sum_{i=1}^{N_q}\int_{t-\tau_i+\delta_i}^t \! \!Q_i(s)x(s-\delta_i)\ ds+\int_{t}^\infty\sum_{i=1}^{N_p}\bar{P}_i(s)x(s-\tau_i)\, ds
    \end{equation}
    for $t\geq t_1$ for a sufficiently large $t_1$. Now, let 
    \[
    E\coloneqq \{u\in C([t_1-D,\infty),\R^+) : 0\leq u(t)\leq 1,\text{ for all } t\geq t_1-D\}.
    \]
    Define the operator $F$ by
    \begin{equation*}
        (Fu)(t)\!=\!
        \left\{\begin{aligned}
             &\frac{1}{x(t)}
            \Bigg[
                \sum_{i=1}^{N_r} R_i(t)u(t-r_i)x(t-r_i) \\
         &\begin{aligned}
           & \hphantom{\frac{1}{x(t)}\Bigg[}+ \sum_{i=1}^{N_q} \int_{t-\tau_i+\delta_i}^t Q_i(s)u(s-\delta_i)x(s-\delta_i)\,ds \\
           & \hphantom{\frac{1}{x(t)}\Bigg[}     +\int_{t}^{\infty}\sum\limits_{i=1}^{N_p} \bar{P}_i(s)u(s-\tau_i)x(s-\tau_i)\,ds
            \Bigg],& & t_1\le t, \\[10pt]
           &  \frac{t-t_1+D}{D}(Fu)(t_1) + 1-\frac{t-t_1+D}{D}, & &t_1-D \le t < t_1.
            \end{aligned}        
        \end{aligned}\right.
    \end{equation*}
    It is easy to see that $F$ maps $E$ into itself. In addition, for all $u \in E$:  $(Fu)(t)>0$,  for all $t\in [t_1-D,t_1)$.
    Define the sequence $(u_k)_{k\in\N}$, where $u_0\equiv 1$, $u_{k+1}=Fu_k$. By mathematical induction we infer from \eqref{eq:x_large} that $0\leq u_{k+1}\leq u_k\leq1$, for all $k\in\N $. Defining $u(t)\coloneqq \lim_{k\to\infty}u_k(t) $, and using Lebesgue's dominated convergence theorem, we obtain that for $t\geq t_1$
   \begin{align*}
    u(t)
        &= \frac{1}{x(t)}\Bigg[
            \sum_{i=1}^{N_r} R_i(t)\,u(t-r_i)\,x(t-r_i)\\
        &\mathrel{\hphantom{=}}\hphantom{\frac{1}{x(t)}\Bigg[}    
            + \sum_{i=1}^{N_q} \int_{t-\tau_i+\delta_i}^{t}
                Q_i(s)\,u(s-\delta_i)\,x(s-\delta_i)\,ds \\
        &\mathrel{\hphantom{=}}\hphantom{\frac{1}{x(t)}\Bigg[}
            + \int_{t}^{\infty}
                \sum_{i=1}^{N_p} \bar{P}_i(s)\,u(s-\tau_i)\,x(s-\tau_i)\,ds
        \Bigg],
    \end{align*}
    and
    \[u(t)=\frac{t-t_1+D}{D}u(t_1)+1-\frac{t-t_1+D}{D}>0\]
    for $t_1-D\leq t<t_1$. Define $w=ux$, then $w(t)>0$ for $t_1-D\leq t<t_1$, and 
    \begin{equation}\label{w}
    	\begin{aligned}
        w(t)&=\sum_{i=1}^{N_r}R_i(t)w(t-r_i)+\sum_{i=1}^{N_q}\int_{t-\tau_i+\delta_i}^t Q_i(s)w(s-\delta_i)\, ds\\
        &\mathrel{\hphantom{=}} + \int_t^\infty\sum_{i=1}^{N_p}\bar{P}_i(s)w(s-\tau_i)\, ds
    	\end{aligned}
    \end{equation}
    for $t\geq t_1$. By differentiating we see that $w$ is a nonnegative solution of \eqref{eq:constant_delays}. It is left to show that $w(t)>0$, for all $t\in[t_1-D,\infty)$. Assume that there exists a $t^*\geq t_1$ such that $w(t)>0$  for $t_1-D\leq t <t^*$ and $w(t^*)=0$. Then by \eqref{w} we have
    \begin{align*}
        0=w(t^*)&=\sum_{i=1}^{N_r}R_i(t^*)w(t^*-r_i)+\sum_{i=1}^{N_q}\int_{t^*-\tau_i+\delta_i}^{t^*}Q_i(s)w(s-\delta_i)\,ds\\
        &\mathrel{\hphantom{=}} + \int_{t^*}^\infty\sum_{i=1}^{N_p}\bar{P}_i(s)w(s-\tau_i)\,ds.
    \end{align*}
    In view of the nonnegativity of $w$, $R_i(t^*)=0$ ($1\leq i\leq N_r$),  $Q_i(t)= 0$,  for all  $t\in[t^*-\tau_i+\delta_i,t^*]$ and $1\leq i \leq N_q$, moreover, $\bar{P}_i(t)w(t-\tau_i)=0$  for all $t\geq t^* $, and $1\leq i\leq N_p$, which contradicts either \eqref{eq:R+P} or \eqref{eq:P>0fortau}.
\end{proof}
Now that we have a connection between NDDEs and corresponding inequalities, we are in position to prove our second main result.  For convenience of notation, let 
\[\tau_0 \coloneqq \min\limits_{1\leq i\leq N_p}\tau_i,\quad \text{and}\quad r_0\coloneqq \min\limits_{1\leq i\leq N_r}r_i,\]
and introduce the functions 
\begin{equation}\label{eq:Omega_def}
    \Omega_{i,j}(t)\coloneqq \frac{R_j(t-\tau_i)\bar{P}_i(t)}{\bar{P}_i(t-r_j)}
\end{equation}
for $1\leq j\leq N_r$, $1\leq i\leq N_p$ and for all $t\geq t_0+D$, where it makes sense.

\begin{theorem}\label{thm:agar}
    Assume that $\bar{P}_i>0$ for all $1\leq i\leq N_p$, and that there exists a constant $\omega\in(0,\infty)$, such that
\begin{equation}\label{eq:omega_def}
    \omega\leq\min_{i,j}\left\{\inf_{s\geq t_0+D}\Omega_{i,j}(s) \right\}\quad \text{and} \quad \omega<\frac{1}{N_r}.
\end{equation} 
Furthermore, assume
    \begin{equation}
        \label{R+Q<1}
        \sum_{i=1}^{N_r}R_i(t)+\sum_{i=1}^{N_q}\int_{t-\tau_i+\delta_i}^{t}Q_i(s)\,ds\leq1
    \end{equation}
    for all $t\geq t_0+D$, $1\leq i\leq N_p$. If there exists $m\in\N$ such that
    \begin{align*}\label{eq:inf1}\tag{$A_3$}
        \frac{(N_r\omega)^m}{1-N_r\omega}\liminf_{t\to\infty}\sum_{i=1}^{N_p}\int_{t-mr_0-\tau_i}^t\bar{P}_i(s)\,ds&>\frac{1}{e},\\
    \shortintertext{or}
    \label{eq:inf2}\tag{$B_3$}
        \frac{(N_r\omega)^m}{1-N_r\omega}\liminf_{t\to\infty}\sum_{i=1}^{N_p}(mr_0+\tau_i)\bar{P}_i(t)&>\frac{1}{e},\\
    \shortintertext{or}
    \label{eq:sup}\tag{$D_3$}
        \frac{(N_r\omega)^m}{1-N_r\omega}\limsup_{t\to\infty}\int_{t-mr_0-\tau_0}^t\sum_{i=1}^{N_p}\bar{P}_i(s)\,ds&>1,
    \end{align*}
    then every solution of equation \eqref{eq:constant_delays} oscillates.
\end{theorem}

This theorem generalizes and improves \cite[Theorem~3.4.9]{agarwal2004} which only considered the case $N_r=N_p=N_q=1$ and $m=1$.   Observe that if all the delays are constant, $\bar{P}_i(t)>0$ and $\omega>0$ satisfies \eqref{eq:omega_def}, then \Cref{thm:agar} -- already with the choice $m=0$ -- gives more efficient conditions than \Cref{cor:I} (cf.~\Cref{ex:agarwal-type-is-stronger,ex:agarwal-type-is-less-general}). 
\begin{proof} The proof follows a similar argument as that of \cite[Theorem 3.4.9]{agarwal2004}.

    Let $m\in\N$ be such that one of inequalities \eqref{eq:inf1}, \eqref{eq:inf2} and \eqref{eq:sup} holds. For the sake of contradiction, suppose that \eqref{eq:constant_delays} has an eventually positive solution $x$.
    Let $y$ be defined by \eqref{c_eq:y_def} -- recall that here $c_i = \tau_i-\delta_i$, with $\delta_i=0$ for $N_q< i\leq N_p$. Just as in \Cref{c_lem:y_pos}, we obtain 
    \begin{equation}\label{c_eq:const-delay:y'=}
    y'(t)+\sum_{i=1}^{N_p}\bar{P}_i(t)x(t-\tau_i)=0.
    \end{equation}
    Applying \eqref{c_eq:y_def} for $x(t-\tau_i)$ in the above equality we arrive at
    \begin{align*}
        y'(t)+\sum_{i=1}^{N_p}\bar{P}_i(t)y(t-\tau_i)
        &+ \sum\limits_{i=1}^{N_p}\sum_{j=1}^{N_r}\bar{P}_i(t)R_j(t-\tau_i)x(t-\tau_i-r_j) \notag \\
        &+ \sum_{i=1}^{N_p}\sum_{j=1}^{N_q}\bar{P}_i(t)\int_{t-\tau_j+\delta_j}^{t}Q_j(s-\tau_i)x(s-\tau_i-\delta_j)=0.
    \end{align*}
    Omitting the last term and using the definition of $\Omega_{i,j}$ we obtain the inequality
    \[
    y'(t)+\sum_{i=1}^{N_p}\bar{P}_i(t)y(t-\tau_i)
    + \sum_{i=1}^{N_p}\sum_{j=1}^{N_r}\Omega_{i,j}(t)\bar{P}_i(t-r_j)x(t-\tau_i-r_j)\leq 0.
    \]
    Using the definition of $\omega$, we have
    \[
    y'(t)+\sum_{i=1}^{N_p}\bar{P}_i(t)y(t-\tau_i)
    + \omega\sum_{i=1}^{N_p}\sum_{j=1}^{N_r}\bar{P}_i(t-r_j)x(t-\tau_i-r_j)\leq 0,
    \]
    and from equation \eqref{c_eq:const-delay:y'=} it is easy to see that
    \[\sum\limits_{j=1}^{N_r}y'(t-r_j)=-\sum\limits_{j=1}^{N_r}\sum\limits_{i=1}^{N_p}\bar{P}_i(t-r_j)x(t-\tau_i-r_j).\]
    Combining these we obtain
    \[
    \left[y(t)-\omega\sum_{i=1}^{N_r}y(t-r_i)\right]'+\sum_{i=1}^{N_p}\bar{P}_i(t)y(t-\tau_i)\leq 0.
    \]
    
    Now, \Cref{thm:eq-ineq} implies the existence of an eventually positive function $\tilde{y}$ that satisfies
    \begin{equation}\label{eq:reduced_neutral}
        \left[\tilde{y}(t)-\omega\sum_{i=1}^{N_r}\tilde{y}(t-r_i)\right]'+\sum_{i=1}^{N_p}\bar{P}_i(t)\tilde{y}(t-\tau_i)=0.
    \end{equation}
    Let us introduce
    \[z(t)\coloneqq \tilde{y}(t)-\omega\sum\limits_{i=1}^{N_r}\tilde{y}(t-r_i),\]
    from which we have
    \[
        \tilde{y}(t)=z(t)+\omega\sum_{i=1}^{N_r}\tilde{y}(t-r_i).
    \]
    By iteration, we obtain 
    \begin{align*}
        \tilde{y}(t)=&z(t)+\omega\sum_{i_1=1}^{N_r}z(t-r_{i_1})+\omega^2\sum_{i_1=1,i_2=1}^{N_r}z(t-r_{i_1}-r_{i_2})+\dots \\ 
        &+\omega^K\sum_{i_1=1,\dots,i_K=1}^{N_r}z\left(t-\sum_{l=1}^{K}r_{i_l} \right)+ \omega^{K+1}\sum_{i_1=1,\dots,i_{K+1}=1}^{N_r}\tilde{y}\left(t-\sum_{l=1}^{K+1}r_{i_l}\right),
    \end{align*}
    where $K>m$  and $t$ is sufficiently large. By virtue of \Cref{c_lem:y_pos} we have that $z'(t)\leq 0$ and $z(t)>0$. Using the monotonicity of $z$ and the positivity of $\tilde{y}$, we obtain the inequality
    \[
        \tilde{y}(t)\geq \sum_{i=0}^K (N_r\omega)^i z(t-ir_0).
    \]
    Now, omitting the first $m$ terms, and using the monotonicity again,
    and $N_r\omega\in(0,1)$, we infer
    \[
        \tilde{y}(t)\geq (N_r\omega)^m\frac{1-(N_r\omega)^{K-m+1}}{1-N_r\omega}z(t-mr_0).
    \]
    Combining this with equation \eqref{eq:reduced_neutral} we conclude that $z(t)$ is an eventually positive solution of the differential inequality
    \[
        z'(t)+\sum_{i=1}^{N_p}\tilde{p}_i(t) z(t-mr_0-\tau_i)\leq0,
    \]
    where 
    \[
    \tilde{p}_i(t)\coloneqq (N_r\omega)^m\frac{1-(N_r\omega)^{K-m+1}}{1-N_r\omega}\bar{P}_i(t).
    \]
    Applying \cite[Corollary 3.2.2]{gyori1991} we obtain that the differential equation
    \begin{equation}\label{eq:fineq}
        v'(t)+\sum_{i=1}^{N_p}\tilde{p}_i(t)v(t-mr_0-\tau_i)=0
    \end{equation}
    also has an eventually positive solution.
    On the one hand, for large enough $K$, \eqref{eq:inf1}, \eqref{eq:inf2} and \eqref{eq:sup} imply 
    \begin{align}\label{eq:fininf1}
        \liminf_{t\to\infty}\sum_{i=1}^{N_p}\int_{t-mr_0-\tau_i}^t \tilde{p}_i(t)\,ds&>\frac{1}{e},\\
        \label{eq:fininf2}
        \liminf_{t\to\infty}\sum_{i=1}^{N_p}(mr_0+\tau_i)\tilde{p}_i(t)&>\frac{1}{e},\\
    \shortintertext{and}
    \label{eq:finsup}
        \limsup_{t\to\infty}\int_{t-mr_0-\tau_0}^t\sum_{i=1}^{N_p}\tilde{p}_i(t)\,ds&>1,
    \end{align}
    respectively.
    On the other hand, it is well-known, that each of the inequalities \eqref{eq:fininf1}, \eqref{eq:fininf2} and \eqref{eq:finsup} is sufficient for the oscillation of all solutions of equation \eqref{eq:fineq} (see \cite[Corollary 1]{bingtuan1996}, \cite[Corollary 2.6.2]{Erbe1995} and \cite[Remark 2.7.3]{ladde_lak1987}, respectively). This is a contradiction that completes our proof.
\end{proof}

It is worth noting that the conditions in the previous theorem can be significantly sharpened by replacing the limit inferior with limit superior, provided the coefficient functions fulfill an extra condition related to the slowly varying property. We conclude our results by formulating this in the next corollary.

\begin{corollary}\label{cor:slow-osc}
    Assume that inequalities \eqref{eq:omega_def}--\eqref{R+Q<1} hold, the functions $\bar{P}_i$ are bounded,  uniformly continuous, and satisfy $\liminf_{t\to\infty}\int_{t-\tau_i}^t\bar{P}_i(s)\,ds>0 $ for all $1\leq i\leq N_p$. For $m\in\N$, define
    \begin{align*}
        A_m(t)&\coloneqq \frac{(N_r\omega)^m}{1-N_r\omega}\sum_{i=1}^{N_p}\int_{t-mr_0-\tau_i}^{t}\bar{P}_i(s)\,ds,\\
        \shortintertext{and}
            B_m(t)&\coloneqq \frac{(N_r\omega)^m}{1-N_r\omega}\sum_{i=1}^{N_p}\bar{P}_i(t)(mr_0+\tau_i).
    \end{align*}
    If one of the following conditions is fulfilled, then every solution of equation \eqref{eq:constant_delays} oscillates:
    \[
        \text{$A_m$ is slowly varying at infinity, and}\quad \limsup_{t\to\infty}A_m(t)>\frac{1}{e} \quad \text{for some $m\in\N$};
    \]
    \[
        \text{$B_m$ is slowly varying at infinity, and}\quad \limsup_{t\to\infty}B_m(t)>\frac{1}{e} \quad \text{for some $m\in\N$}.
    \]
\end{corollary}
\begin{proof}
    The indirect proof proceeds in a similar way as in \Cref{thm:agar}. The only difference is that we  apply \cite[Theorem 2.2 (a)--(b)]{abel2020} at the end to arrive at a contradiction.
\end{proof}

\section{Examples}\label{sec:examples}
In this section we give several examples illustrating the applicability of our results.

In our first example, \Cref{c_thm:ndde-rdde} can be applied directly.
\begin{example}\label{ex-variable-delays}
    Consider the equation
    \begin{equation*}\label{eq:pl5}
    \biggl[x(t)-\frac{1}{4}x\Bigr(t-\frac{1}{2}\Bigr)\biggr]'
    +2x\biggl(t-\Bigl(\frac{1}{2}\cos t+1+e^{-t+1+\frac{1}{2}\cos t}\Bigr)\biggr)
    -e^{-t}x\bigl(t-e^{-t}\bigr)=0
    \end{equation*}
    for $t\geq 0$. Here we take $N_r=N_p=N_q=1$, $R_1(t)=\frac{1}{4}$, $P_1(t)=2$, $Q_1(t)=e^{-t}$, $r_1(t)=\frac{1}{2}$, $\tau_1(t)=\frac{1}{2}\cos t+1+e^{-t+1+\frac{1}{2}\cos t}$, and $\delta_1(t)=e^{-t}$. The associated implicit relation yields the explicit solution $c_1(t)=\tfrac{1}{2}\cos t+1$. One can show that $\bar{P}_1(t)=2-e^{c_1(t)-t}(1-c_1'(t))$, which tends to $2$ as $t\to\infty$. Hence, there exists $T_0>0$ such that $\bar{P}_1(t)\ge 1$ for all $t\ge T_0$. Applying Corollary~\ref{cor:I}, we obtain
    \[ \liminf_{t\to\infty}\int_{t-\tau_1(t)}^{t}\bar{P}_1(s)\,ds \ge \liminf_{t\to\infty}\tau_1(t) \ge \liminf_{t\to\infty}c_1(t) = \frac{1}{2} > \frac{1}{e},
    \]
    since $c_1(t)\le\tau_1(t)$. Therefore, every solution oscillates.
\end{example}
 
 The next example illustrates the applicability of \Cref{thm:agar} in case of constant delays, and the importance of the appropriate choice of $m$.
\begin{example}
    Consider the NDDE
    \begin{equation*}\label{eq:pl1}
        \bigg[x(t)-\frac{1}{2}x\Bigl(t-\frac{\pi}{2}\Bigr)\bigg]'+\bigg(\frac{3}{4}\sin(4t)+\frac{5}{4}\bigg)x(t-0.1)+\frac{1}{10}x(t-0.01)-\frac{1}{4}x(t-0.05)=0
    \end{equation*}
    for $t\geq 0$. This equation is of the form \eqref{eq:constant_delays} with $N_r=N_q=1$, $N_p=2$, 
    $R_1(t)=\frac{1}{2},\ P_1(t)=\frac{3}{4}\sin(4t)+\frac{5}{4}$, $P_2(t)= \frac{1}{10}$, $Q_1(t)=\frac{1}{4}$, $r_1=\frac{\pi}{2}$, $\tau_1=0.1$, $\tau_2=0.01$ and $ \delta_1=0.05 $, from which we obtain
    $\bar{P}_1(t)=\frac{3}{4}\sin(4t)+1,\ \bar{P}_2(t)=\frac{1}{10}$ and $\Omega_{1,1}(t)= \Omega_{2,1}(t)=\frac{1}{2}$. A suitable choice for $\omega$ is $\frac{1}{2}$.
    Then the left-hand side of condition \eqref{eq:inf2} reads
\[
2^{1-m} \bigg[\Bigl(\frac{m\pi}{2}+0.1\Bigr)\frac{1}{4}+\Bigl(\frac{m\pi}{2}+0.01\Bigr)\frac{1}{10}\bigg].
\]
Choosing $m=2$, we have
\[
\frac{1}{2} \bigg(\frac{\pi+0.1}{4}+\frac{\pi+0.01}{10}\bigg) \approx 0.563 > \frac{1}{e}.
\]
Thus, condition \eqref{eq:inf2} is satisfied, and hence every solution oscillates. 
We note that one could also choose $m=1$ or $m=3$ to draw the same conclusion, 
while other values of $m$ do not allow the oscillatory behavior to be determined via \eqref{eq:inf2}. 
For comparison, condition \eqref{eq:sup} provides a sufficient condition for oscillation with $m\in \{1,2,3\}$, as well, while \eqref{eq:inf1} guarantees oscillation if we set $1\leq m\leq 5$.
\end{example}
\medskip

In the following example we take advantage of the slowly varying property of the coefficient functions.

\begin{example} Consider the equation
    \begin{align*}\label{eq:pl2}
        &\bigg[x(t)-\frac{1}{2}x(t-0.15)\bigg]'+1.1x(t-0.05)\\
                &\qquad +(2\cos(\ln t)+2.05)x(t-0.01)-x(t-0.02)=0,
    \end{align*}
    for $t\geq 1$. This NDDE is of the form \eqref{eq:constant_delays} with $N_r=N_q=1$, $N_p=2$, 
    $R_1(t)= \frac{1}{2}$, \mbox{$P_1(t)= 1.1$}, $P_2(t)=\bar{P}_2(t)=2\cos(\ln t)+2.05$, 
    $\ Q_1(t)= 1$, $r_1=0.15,\ \tau_1=0.05,\ \tau_2=0.01,\ \delta_1=0.02,$ $\bar{P}_1(t)= 0.1$. Furthermore, 
    $\Omega_{1,1}(t)=\frac{1}{2}$ and  $\Omega_{2,1}(t)
     \to \frac{1}{2} $ as $t\to\infty$, therefore we can choose $\omega$ to be smaller than, but arbitrarily close to $\frac{1}{2}$.
    Even with the choice $\omega=\frac{1}{2}$, inequality \eqref{eq:inf2} is not fulfilled for any $m\in \N$, as
    \begin{multline*}
         \frac{(\frac{1}{2})^m}{1-\frac{1}{2}}\liminf_{t\to\infty} \bigg[0.1(0.15m+0.05)+\left(2\cos(\ln(t))+2.05 \right)(0.15m+0.01) \bigg]\\
        =2^{1-m}(0.0225m+0.0055) < \frac{1}{e}\quad \text{for all } m\in\N.
    \end{multline*}
    Notice that thanks to the slow variance of the functions $\bar{P}_1$ and $\bar{P}_2$, conditions \eqref{eq:inf1} and \eqref{eq:inf2} are equivalent, so \eqref{eq:inf1} cannot guarantee the oscillation of all solutions, either.
    
    On the other hand, these functions have positive lower bounds, so we may try to apply \Cref{cor:slow-osc}. With $\omega=\frac {1}{2}$ and $m=1$, one would have
    \begin{equation*}
        \limsup_{t\to\infty}B_1(t) =0.668>\frac 1e
    \end{equation*}
    Hence, by choosing $\omega$ slightly below $\frac{1}{2}$, we may conclude that every solution oscillates. The same conclusion follows if one chooses $m=2$ or $m=3$. However, a similar simple calculation shows that $\limsup_{t\to \infty} B_m(t)<1$ for all $m\in\N$ and all $\omega<\frac{1}{2}$, thus \eqref{eq:sup} is not satisfied, and \Cref{thm:agar}  cannot be applied here.
\end{example}
\medskip

Finally, one might wonder whether \Cref{cor:I} or \Cref{thm:agar} is stronger for constant delays, but there is no clear answer to that. The following two examples will illustrate this fact.
\begin{example} \label{ex:agarwal-type-is-less-general} Consider the equation 
    \begin{equation*}\label{eq:pl3}
        \bigg[x(t)-e^{-t} x(t-2\pi)\bigg]'+3x(t-1)+(\sin t+1.5)x(t-2)-x\biggl(t-\frac{1}{2}\biggr)=0,\qquad t\geq 1
    \end{equation*}
    with constant delays. This equation is in the form of \eqref{eq:constant_delays} with $N_r=N_q=1$, $N_p=2$,
    $R_1(t)=e^{-t},\ P_1(t)= 3,\ P_2(t)=\sin(t)+1.5$, $ Q_1(t) = 1$, 
    $r_1=2\pi,\ \tau_1=1,\ \tau_2=2,\ \delta_1=\frac{1}{2}$, from which  $\bar{P}_1(t)=2,\ \bar{P}_2(t)=\sin(t)+1.5$. On the one hand,  $\Omega_{1,1}(t)=\Omega_{2,1}(t)=e^{-t}\to0 $  as $t$ goes to infinity, so we cannot apply \Cref{thm:agar}. On the other hand, if we let $\PP_i=\bar{P}_i$ and $\cc_i=\tau_i-\delta_i$ $(i=1,2)$, then all the conditions of \Cref{cor:I} are satisfied. Hence, the left-hand side of the inequality \eqref{A_1} takes the following form:
    \[
        \liminf_{t\to\infty}\bigg(\int_{t-1}^t2\,ds + \int_{t-2}^{t}\sin(s)+1.5\,ds\bigg)>3,
    \]
    which is clearly larger than $\frac{1}{e}$, hence \Cref{cor:I} implies that every solution oscillates.
\end{example}

\begin{example}\label{ex:agarwal-type-is-stronger}
Let us consider the NDDE
    \begin{equation*}\label{eq:pl4}
        \bigg[x(t)-\frac{1}{3}x(t-2)\bigg]'+(\cos2t+2)x\biggl(t-\frac{1}{2}\biggr)-\biggl[\cos\biggl(2t+\frac{1}{2}\biggr)+\frac{3}{2}\biggr]x\biggl(t-\frac{1}{4}\biggr)=0,
    \end{equation*}
    which is of the form \eqref{eq:constant_delays} with $N_r=N_p=N_q=1$, $R_1(t)=\frac{1}{3},\ P_1(t)=\cos2t+2,\ Q_1(t)=\cos(2t+\frac{1}{2})+\frac{3}{2},$\ 
    $r_1=2,\ \tau_1=\frac{1}{2},\ \delta_1=\frac{1}{4},\ c_1=\frac{1}{4} ,\ \bar{P}_1(t)=\frac{1}{2},\ \Omega_{1,1}(t)=\frac{1}{3}$, thus $\omega\coloneqq \frac{1}{3}$.
    Substituting into the left-hand side of any of the conditions in \eqref{A_1}--\eqref{D_1} in \cref{cor:I},  we obtain $\frac{1}{4}$, which is smaller than $\frac{1}{e}$. Thus, \Cref{cor:I} does not imply oscillatory behavior. However,  condition \eqref{eq:inf1} in \Cref{thm:agar} is fulfilled with $m=2$ (it boils down to $0.375>\frac{1}{e}\approx 0.368$), thus \Cref{thm:agar} implies that every solution oscillates.
\end{example}

\section*{Acknowledgment}
This research was supported by the National Research, Development and Innovation (NRDI) Fund, Hungary, [project no.\ FK~142891] and the National Laboratory for Health Security [RRF-2.3.1-21-2022-00006].  Á.\;G.\ was supported by the János Bolyai Research Scholarship of the Hungarian Academy of Sciences.

\end{document}